\documentclass{amsart}
\usepackage[utf8]{inputenc}
\usepackage{amsthm,amsmath,amssymb, amscd}
\usepackage[colorlinks=true,allcolors=blue!80!black]{hyperref}
\usepackage[margin=1.25in]{geometry}
\usepackage[capitalise]{cleveref}
\usepackage[parfill]{parskip} 
\usepackage{graphics}
\usepackage{tikz, tikz-cd}
\usetikzlibrary{matrix, snakes, patterns, shadows.blur, backgrounds}
\usepackage[all]{xy}
\usepackage{enumerate}
\usepackage{bbm}
\usepackage{stmaryrd}
\usepackage{nicematrix}
\usepackage{dsfont}
\usepackage{caption}
\usepackage{subcaption}
\usepackage{import}

\theoremstyle{theorem}
\newtheorem{thm}{Theorem}
\newtheorem{prop}[thm]{Proposition}
\newtheorem{lem}[thm]{Lemma}
\newtheorem{cor}[thm]{Corollary}

\theoremstyle{definition}
\newtheorem{defn}[thm]{Definition}

\DeclareMathOperator\Diff{Diff}
\DeclareMathOperator\Homeo{Homeo}

\DeclareMathOperator\emb{Emb}
\DeclareMathOperator\Map{Map}

\DeclareMathOperator*\colim{colim}

\DeclareMathOperator\BSO{BSO}

\newcommand{\interior}[1]{\smash{\mathring{#1}}}
\DeclareMathOperator\sep{Sep} 
\newcommand{\septop}{\sep^{\rm top}}

\newcommand{\dsep}{\operatorname{DSep}^{\rm top}}

\newcommand{\ca}{\!\mid\!} 

\newcommand{\Sing}{\mathrm{S}_\bullet}
\newcommand{\SingHomeo}{\mathrm{S}_\bullet\!\Homeo} 
\newcommand{\D}{\mathrm{D}}

\newcommand{\hq}{/\!\!/} 

\newcommand{\Top}{\mathrm{Top}}

\newcommand{\BHomeo}{B\!\Homeo}
\newcommand{\BDiff}{B\!\Diff}

\newcommand{\bbQ}{\mathbb{Q}}

\newcommand{\Z}{\mathbb{Z}}
\newcommand{\bt}{\bullet}

\newcommand{\fakeenv}{} 
\newenvironment{restate}[2]  
{
  \renewcommand{\fakeenv}{#2} 
  \theoremstyle{plain}
  \newtheorem*{\fakeenv}{#1~\ref{#2}} 
  \begin{\fakeenv}  
}
{
  \end{\fakeenv}
}

\title
{Extending families of homeomorphisms over 4-dimensional handlebodies}

\author{Rachael Boyd}
\address{School of Mathematics and Statistics, University of Glasgow, Glasgow G12 8QQ, UK}
\email{rachael.boyd@glasgow.ac.uk}
\urladdr{https://www.maths.gla.ac.uk/~rboyd/} 

\author{Corey Bregman}
\address{Department of Mathematics, Tufts University, Medford, MA 02155, USA}
\email{corey.bregman@tufts.edu}
\urladdr{https://sites.google.com/view/cbregman}

\author{Jan Steinebrunner}
\address{Gonville \& Caius College,
University of Cambridge,
Cambridge, 
UK}
\email{js2675@cam.ac.uk}
\urladdr{https://www.jan-steinebrunner.com}


\begin{document}

\newpage
\maketitle
\begin{abstract}
   Let $H_g$ denote the  4-dimensional handlebody of genus $g$ and $U_g$ its boundary. We show that for all $g \ge 0$ the map from $\BHomeo(H_g)$ to $\BHomeo(U_g)$ induced by restriction to the boundary admits a section.
\end{abstract}

\section*{Introduction}
Let $U_g=(S^1\times S^2)^{\# g}$ be the boundary of the $4$-dimensional handlebody $H_g=(S^1\times D^3)^{\natural g}$, and let $\Homeo(M)$ be the group of homeomorphisms (not required to be orientation preserving) of a manifold $M$.
Our main result is the following.

\begin{restate}{Theorem}{thm:section}
    For all $g \ge 0$ the map
    $\BHomeo(H_g) \overset{\partial}{\to}\BHomeo(U_g)$
    has a section.
\end{restate}

We now discuss some immediate consequences of the theorem.
Since $\BHomeo(M)$ classifies topological $M$ bundles, this shows that any topological $U_g$ bundle can be fiberwise filled by a topological $H_g$ bundle. 
In fact, this can also be applied to families of 4-manifolds with $U_g$ boundary, via the following corollary.

\begin{restate}{Corollary}{cor:section}
    Let $M$ be a compact $4$-manifold with $U_g \subseteq \partial M$.
    Then the restriction map
    \[
        \BHomeo(M \cup_{U_g} H_g, H_g) \longrightarrow \BHomeo(M, U_g)
    \]
    admits a section.
\end{restate}

Laudenbach--Poenaru \cite{LaudenbachPoenaru1972} proved that every diffeomorphism of $U_g$ can be extended to a diffeomorphism of $H_g$.
It follows from local connectedness results due independently to Fisher \cite{Fisher}, Kister \cite{Kister}, Hamstrom \cite{Hamstrom}, and approximation results of Moise and Bing \cite{Moise54,Bing54}, that the same holds for homeomorphisms.
Our theorem extends this result to families. With this viewpoint, there is scope to apply \cref{cor:section} to study families of homeomorphisms of 4-manifolds obtained via Kirby diagrams (see \cite[Section 3]{NiuThesis}) or trisections (see, e.g.~\cite[Gay, p.~240]{OWR-Morphisms-in-low-dimensions}). 
Of course, for these settings it would be more natural to use a section for diffeomorphisms.
 
Ideally we would like to loop the  section in \Cref{thm:section} to obtain a section $\Omega s$ on the level of topological groups, but for a 3-manifold $M$, it is not known that $\Homeo(M)$ has the homotopy type of a CW complex and so we only have a section after taking CW replacements, \emph{i.e.}~after passing to $|\SingHomeo(-)|$.
In particular, this yields a weak equivalence
\[
    \Homeo_\partial(H_g) \times \Homeo(U_g) \simeq \Homeo(H_g).
\]

Restricting to mapping class groups we obtain the following corollary:
\begin{align*}
    \pi_0\Homeo(H_g) 
    & \cong \pi_0\Homeo(U_g) \ltimes \pi_0 \Homeo_\partial(H_g) \\
    & \cong \left(\mathrm{Out}(F_g) \ltimes (\Z/2)^g\right) \ltimes \pi_0 \Homeo_\partial(H_g)
\end{align*}
where the second isomorphism is Brendle--Broaddus--Putman \cite[Theorem A]{Brendle2023}.
Recall that Budney--Gabai showed that $\pi_0\Homeo_\partial(H_1)$ is not finitely generated~\cite{BudneyGabai2023} 
and thus we conclude $\pi_0 \Homeo(H_1)$ is not finitely generated. 
(This can already be deduced from \cite{BudneyGabai2023} using work of Hatcher \cite{hatcher1981,Hatcher1981revised}, but perhaps for higher genus our perspective could be utilised.)

The section from \Cref{thm:section} allows us to define rational characteristic classes for $U_g$-bundles analogous to the Miller--Morita--Mumford classes for surfaces.
(Recall that the usual generalised MMM-classes $\kappa_c$ are all trivial in dimension $3$ by a theorem of Ebert \cite{Ebert2013}.)
For every monomial $c = p_1^a e^b$ in $H^*(\BSO(4);\bbQ) \cong \bbQ[p_1,e]$ and every $H_g$ fiber bundle $\pi\colon E \to B$ over a compact manifold $B$ we can define
$c(\mathfrak{t}_v E) \in H^{4(a+b)}(E; \bbQ)$ where $\mathfrak{t}_vE$ is the vertical tangent microbundle, as in \cite[\S4]{Ebert-RW-generalised-MMM}.
Taking the Becker--Gottlieb transfer yields
\(
    \pi^! c(\mathfrak{t}_v E) \in H^{4(a+b)}(B; \bbQ)
\).
By varying $B$ as in \cite[Proposition 4.2]{Ebert-RW-generalised-MMM} these classes assemble into a characteristic class 
\[
    \alpha_c \in H^{4(a+b)}(\BHomeo(H_g); \bbQ)
\]
and pulling it back under the section from \Cref{thm:section} we get
\(
    s^*\alpha_c \in H^{4(a+b)}(\BHomeo(U_g); \bbQ).
\)
Hatcher proposed during a talk \cite{Hatcher2012} that the stable cohomology of $\BHomeo_{D^3}(U_g)$ as $g \to \infty$ should be given by
\[
    H^*(\colim_{g \to \infty} \BHomeo_{D^3}(U_g); \bbQ) 
    \stackrel{\text{conj.}}{\cong} 
    H^*(\Omega^\infty \Sigma^\infty_+ \BSO(4); \bbQ) 
    \cong 
    \bbQ\left[ \rho_c \;|\; c \text{ monomial in } p_1 \text{ and } e\right].
\]
The above construction of $s^*\alpha_c$ yields reasonable candidates for the restriction of these generators $\rho_c$ to $\BHomeo(U_g)$ for all $g$.
Hatcher's approach to computing the stable homology implicitly uses that a section as in \Cref{thm:section} exists, in order to obtain the description of the stable homology above.

\subsection*{Outline of paper} The proof uses the space of separating systems developed in previous work of the authors \cite{BoydBregmanSteinebrunner-finiteness}, which studied $\BDiff(M)$, imported into the topological setting. This setting involves passing to a singular set-up and we introduce the necessary theory in \Cref{sec:simplicialprelims}. The topological separating systems are the topic of \Cref{sec:topsep}, and the theorem is proved in \Cref{sec:proof}.

\subsection*{Acknowledgements}
We would like to thank Mark Powell and Oscar Randal-Williams for helpful comments and conversations.
During this work the first author was supported by EPSRC Fellowship EP/V043323/2. 
The second author was supported by NSF grant DMS-2401403. 
The third author was supported by the Independent Research Fund Denmark (grant no.~10.46540/3103-00099B)
and the Danish National Research Foundation through the ‘Copenhagen Centre for Geometry and Topology’ (grant no.~CPH-GEOTOP-DNRF151).

We would like to thank the Isaac Newton Institute for funding the satellite programme \emph{Topology, representation theory and higher structures} based at Gaelic College, Sabhal Mòr Ostaig, Isle of Skye, where some of this work was carried out.

\section{Simplicial preliminaries}\label{sec:simplicialprelims}
For $X$ a topological space, let $\Sing X$ denote the singular simplicial set defined via $\mathrm{S}_n X = \Map_{\Top}(\Delta^n, X)$.
In particular, for $\Homeo(M)$ the group of homeomorphisms $\varphi\colon M \to M$, we let $\SingHomeo(M)$ denote the resulting simplicial group.
We can alternatively think of an $n$-simplex in $\mathrm{S}_n\Homeo(M)$ as a homeomorphism $\Delta^n \times M \to \Delta^n \times M$ that commutes with the projection to $\Delta^n$. Face and degeneracy maps are inherited from~$\Delta^n$.

We will work with these simplicial sets instead of topological spaces, hence we will encounter bi-simplicial sets whenever one would otherwise consider simplicial (topological) spaces.
The geometric realisation of a simplicial (topological) space is then replaced by taking the diagonal bisimplicial set $\delta(Y_{\bt, \bt})$ defined by $\delta(Y)_n = Y_{n,n}$.
For example, if a simplicial group $G$ acts on a simplicial set $X$, then the bar construction
\[
    \mathrm{Bar}_n(X_m, G_m, *) = X_m \times (G_m)^{\times n}
\]
is a bisimplicial set, and we define the homotopy orbit construction to be its diagonal
\[
    X\hq G := \delta( \mathrm{Bar}_\bt(X_\bt, G_\bt, *) ).
\]
The realisation of the diagonal $|\delta(Y)|$ is homeomorphic to the realisation of the simplicial space $[n] \mapsto |Y_{\bt,n}|$.
In fact, $|\delta(Y)|$ is also weakly equivalent to the more well-behaved fat geometric realisation of $[n] \mapsto |Y_{\bt,n}|$ \cite[Theorem 7.1 and Lemma 1.7]{EbertRandalWilliams}.
Therefore
\[
    |X \hq G| \simeq \|[n] \mapsto \mathrm{Bar}_n(|X|, |G|, *) \| = |X|\hq |G|.
\]
The simplicial set $EG = G \hq G$ is contractible and admits a (left) $G$-action that is level-wise free.
We can alternatively write the homotopy orbit construction as the quotient of the diagonal action
\[  
    X \times EG \longrightarrow (X \times EG)/G \cong X \hq G.
\]
This map is a Kan fibration with fiber $G$ \cite[Lemma 18.2]{May1992simplicial}.

For a topological group $H$ (such as $\Homeo(M)$ or $\Diff(M)$) we can find an $H$ principal bundle $EH \to BH$ such that $BH$ is a CW complex and $EH$ is weakly contractible (take Milnor's construction \cite{Milnor1956} and pull back along a CW approximation).
This $BH$ classifies $H$-principal bundles over spaces homotopy equivalent to CW complexes, and is unique up to homotopy equivalence.
$\Sing H$ is a simplicial group and there are equivalences
\[
    |* \hq \Sing H| \simeq * \hq |\Sing H| = B(|\Sing H|) \simeq BH.
\]
Using this model, $\Omega BH$ is always weakly homotopy equivalent to $H$ and they are homotopy equivalent when $H$ has the homotopy type of a CW complex.

We will need the following lemma for comparing homotopy orbit constructions, which is a simplicial analogue of \cite[Lemma 2.10]{BoydBregmanSteinebrunner-finiteness}.
\begin{lem}[Orbit stabiliser lemma]\label{lem:orbit-stabiliser}
     Let $\varphi\colon G \to H$ be a simplicial group homomorphism and $f\colon X \to Y$ a map of simplicial sets such that $G$ acts on $X$, $H$ acts on $Y$, and $f$ is $G$-equivariant.
     Assume that for each $y \in Y_0$ there is an $h \in H_0$ and $x \in X_0$ such that $h\cdot f(x)$ is in the same path component as $y$. Assume further that one of the following two conditions holds.
     \begin{enumerate}
         \item For each $x \in X_0$ the commutative square
         \[\begin{tikzcd}
             G \ar[r, "{-\cdot x}"] \ar[d, "\varphi"'] & X \ar[d, "f"] \\
             H \ar[r, "{-\cdot f(x)}"] & Y 
         \end{tikzcd}\]
         is a homotopy pullback square.
         \item
         The group actions are such that for all $x \in X_0$ the maps $-\cdot x\colon G \to X$ and $-\cdot f(x)\colon H \to Y$ are Kan fibrations and the induced map on stabilisers
         \[
            \mathrm{Stab}_G(x) \longrightarrow \mathrm{Stab}_H(f(x))
         \]
         is a weak equivalence.
     \end{enumerate}
     Then the induced map on homotopy orbits is a weak equivalence
     \[
            X\hq G \longrightarrow  Y\hq H.
     \]
\end{lem}
\begin{proof}
    First, note that (2) implies (1).
    Indeed, (2) says that the horizontal maps in (1) are Kan fibrations, so their fibers (which are exactly the stabiliser groups) are equivalent to their homotopy fibers.
    A square is a homotopy pullback square if and only if the induced map on horizontal fibers at every base point is a weak equivalence, but this is exactly what the hypothesis about stabiliser groups in (2) tells us.

    Now suppose we have (1). We may additionally assume that $X$ and $Y$ are Kan.
    In order to prove the lemma, we need to show that for each $[y] \in Y\hq H$ the homotopy fiber of $f \hq \varphi$ at $[y]$ is contractible.
    By the additional hypothesis, $[y]$ is in the same path component as $[f(x)]$ for some $x \in X_0$, so it suffices to study the homotopy fiber at $[f(x)]$
    
    If we replace $X$ by $X \times EG$ and $Y$ by $Y \times EH$, then the square remains a homotopy pullback square and fits into a map of fiber sequences
    \[\begin{tikzcd}
    	{\Omega(X \hq G)} & G & {X \times EG} & {X\hq G} \\
    	{\Omega(Y \hq H)} & H & {Y \times EH} & {Y \hq H.}
    	\arrow[dashed, from=1-1, to=1-2]
    	\arrow["{\Omega(f \hq \varphi)}"', "\simeq", from=1-1, to=2-1]
    	\arrow[from=1-2, to=1-3]
    	\arrow["\varphi"', from=1-2, to=2-2]
    	\arrow["\lrcorner^h"{anchor=center, pos=0.125}, draw=none, from=1-2, to=2-3]
    	\arrow[from=1-3, to=1-4]
    	\arrow["{f \times E\varphi}", from=1-3, to=2-3]
    	\arrow["{f\hq \varphi}", from=1-4, to=2-4]
    	\arrow[dashed, from=2-1, to=2-2]
    	\arrow[from=2-2, to=2-3]
    	\arrow[from=2-3, to=2-4]
    \end{tikzcd}\]
    Here we have continued the fiber sequence to the left by looping the map $f \hq \varphi$.
    Since the middle square is a homotopy pullback square the induced map $\Omega(f \hq \varphi)$ on the horizontal homotopy fibers is an equivalence.
    This shows that $f \hq \varphi$ is an isomorphism on $\pi_i$ for $i \ge 1$.

    It follows from the additional hypothesis that $f \hq \varphi$ is surjective on path components, so it remains to show that $\pi_0(f \hq \varphi)$ is injective.
    Suppose $[x], [x'] \in \pi_0(X\hq G)$ (with representatives $x,x' \in X_0$) are such that $[f(x)] = [f(x')] \in \pi_0(Y \hq H)$.
    This means that there is $h \in H_0$ and a path $\gamma$ from $h\cdot f(x)$ to $f(x')$.
    The triple $(h, \gamma, x')$ defines a vertex in the homotopy pullback $H \times_Y^h X$, where we take $H \to Y$ to be the map that acts on $f(x)$.
    Since the map $G \to H \times_Y^h X$ (given by $g \mapsto (\varphi(g), \mathrm{const}_{g\cdot f(x)}, g\cdot x)$) is assumed to be a weak equivalence, we can find $g \in G_0$ such that $(\varphi(g), \mathrm{const}_{g\cdot f(x)}, g\cdot x)$ is in the same path component as $(h, \gamma, x')$.
    In particular, $g\cdot x$ is in the same path component as $x'$, so $[x] = [x'] \in \pi_0(X \hq G)$, proving that $\pi_0(f \hq \varphi)$ is injective.
\end{proof}

\section{Topological separating systems}\label{sec:topsep}
In our previous work, \cite{BoydBregmanSteinebrunner-finiteness}, we study the homotopy type of $\BDiff(M)$ and $\BDiff_\partial(M)$. One of our main tools is a space of \emph{separating systems} for a 3-manifold $M$, which parametrises decompositions of $M$ into irreducible manifolds \cite[\S 3]{BoydBregmanSteinebrunner-finiteness}. This space is denoted $\sep(M)$, and point corresponds to a collection of disjointly embedded spheres $\Sigma\subset \interior{M}$ that `cut' $M$ into (punctured) irreducible pieces. Note that $M \setminus \Sigma$ is diffeomorphic to the interior of a compact 3-manifold with boundary, which we denote as $M \ca \Sigma$. Intuitively, $M \ca \Sigma$ is the manifold obtained from cutting $M$ along $\Sigma$. We endow $\sep(M)$ with a poset structure induced by inclusion of separating systems, and we show that the realisation of the nerve is contractible. Thus we gain a model $\|\sep_\bullet(M)\|\hq \Diff(M)$ for $\BDiff(M)$.

We will build similar models for $\BHomeo(U_g)$ and $\BHomeo(H_g)$. We start by defining a topological version of $\sep(M)$, specifically in the setting of $M=U_g$. Since we will require the action map from $\SingHomeo(U_g)$ to this space to be a Kan fibration, we define everything simplicially. Recall that since a separating system $\Sigma \subset M$ satisfies that $M \ca \Sigma$ is a disjoint union of (punctured) irreducible manifolds, it follows that if $\Sigma\in \sep(U_g)$, then $U_g \ca\Sigma$ is homeomorphic to a disjoint union of punctured $3$-spheres.

The following discussion follows Appendix I of Burghelea--Lashof--Rothenberg \cite{BurgheleaLashofRothenberg1975}, an account of which can also be found in Kupers \cite[\S3.2]{Kupers-homeo}.

\begin{defn}\label{defn:lf-family}
    Let $M$ be a compact topological $n$ manifold.
    A locally-flat $\Delta^k$-family of $d$-dimensional submanifolds in $M$ is a compact subset $W \subseteq \Delta^k \times M$ such that for every $(t,x) \in W$ there is a neighbourhood $t \in B \subseteq \Delta^k$ and an open embedding $\varphi\colon B \times \mathbb{R}^d \hookrightarrow B \times M$ over $B$ satisfying $\varphi(t,0) = (t,x)$ and $\varphi^{-1}(W) = B \times \mathbb{R}^k$.
    If $M$ has boundary we further require $\partial W = W \cap \partial (\Delta^k \times M)$ and we replace $\mathbb{R}^d$ and $\mathbb{R}^k$ by the half-space whenever $(t,x)$ is a boundary point.
\end{defn}

It follows from the definition that $W$ is a $(d+k)$-dimensional manifold and the projection $W \to \Delta^k$ is a trivial fiber bundle with fiber $N$.
This definition is chosen such that the simplicial set $\mathrm{Sub}(M)_\bullet$ whose $k$-simplices are locally flat $\Delta^k$-families of submanifolds of $M$ is isomorphic to
\[
    \mathrm{Sub}(M)_\bullet \cong \coprod_{[N]} \emb^{\rm lf}(N, M)_\bullet / \SingHomeo(N)
\]
where the coproduct runs over representatives of homeomorphism types of submanifolds.
Here the definition of the simplicial set of locally flat embeddings is taken to be as in \cite[Appendix I, p.119]{BurgheleaLashofRothenberg1975}.
Acting by homeomorphisms of $M$ on a fixed submanifold $N \subseteq M$ defines a map
    \[
        \SingHomeo(M) \longrightarrow \emb^{\rm lf}(N, M)_\bullet 
        \longrightarrow \emb^{\rm lf}(N, M)_\bullet/\SingHomeo(N) \subseteq \mathrm{Sub}(M)_\bullet
    \]
The first of these maps is a Kan fibration by the parametrised isotopy extension theorem for locally flat embeddings (\cite[Theorem 4.14, p.129]{BurgheleaLashofRothenberg1975} or \cite[Theorem 3.9]{Kupers-homeo}).
The second map is a Kan fibration by \cite[Lemma 18.2]{May1992simplicial} as the action of $\SingHomeo(N)$ on $\emb^{\rm lf}(N,M)_\bullet$ by pre-composition is level-wise free.
The third map is a Kan fibration because it is an inclusion of path components.
Therefore the composite map is a Kan fibration.

\begin{defn}\label{def: septop as simplicial set}
    A $\Delta^k$-family of separating systems in $U_g$ is a locally flat $\Delta^k$-family $W \subset \Delta^k \times U_g$ of submanifolds such that in each fiber $W_t$ over $t \in \Delta^k$
    \begin{enumerate}
        \item 
        $W_t$ is a disjoint union of $2$-spheres in $U_g$, and
        \item 
        $U_g \ca W_t \cong \sqcup_{i=1}^{l} (S^3 \setminus \sqcup_{m_i} \interior{D}^3)$ for some $l \ge 1$ and $m_i \ge 2$. 
    \end{enumerate}
    Note that $W_t$ is a separating system in the sense of \cite{BoydBregmanSteinebrunner-finiteness}.
    Let $\septop(U_g)_\bt$ be the simplicial set 
    \[
        \septop(U_g)_k = \{W\subset \Delta^k\times M \;|\; W \text{ is a } \Delta^k\text{-family of separating systems}\}.
    \]
    Face and degeneracy maps are inherited from~$\Delta^k$. This is a poset under $\subseteq$ and taking the nerve of this poset yields a bisimplicial set with
    \[
    \septop_n(U_g)_k=\{W_0\subseteq W_1 \subseteq \cdots \subseteq W_n \;|\; W_i \in\septop(U_g)_k \}
    \]
    where the face operators in the nerve direction forget elements of the chain and the degeneracy operators repeat elements.
    We then let $\delta\septop(U_g)_\bullet$ denote the diagonal simplicial set with $\delta\septop(U_g)_n = \septop_n(U_g)_n$.
\end{defn}

In our analogous definition for $H_g$, we cut the manifold up using discs, and require that the restriction to the boundary $U_g$ is a separating system.

\begin{defn}\label{def: dseptop as simplicial set}
    A $\Delta^k$-family of separating systems in $H_g$ is a locally flat $\Delta^k$-family $\D \subset \Delta^k \times H_g$ of submanifolds such that in each fiber $\D_t$ over $t \in \Delta^k$
    \begin{enumerate}
        \item 
        $\D_t$ is a disjoint union of $3$-discs,
        \item 
        $H_g \ca \D_t$ is a disjoint union of $4$-discs, and 
        \item $\D \cap (\Delta^k \times \partial H_g)$ is a $\Delta^k$-family of separating systems in $U_g = \partial H_g$.
    \end{enumerate}
    As in the case of separating systems we let $\dsep(H_g)_\bt$ be the simplicial set
    \[
        \dsep(H_g)_k = \{\D\subset \Delta^k\times M \;|\; \D \text{ is a } \Delta^k \text{-family of disc systems}\},
    \]
    with face and degeneracy maps once again inherited from~$\Delta^k$. Then $\dsep(H_g)_k$ has a poset structure given by $\subseteq$.
    Taking the nerve of this poset yields a bisimplicial set with
    \[
        \dsep_n(H_g)_k=\{\D_0\subseteq \D_1 \subseteq \cdots \subseteq \D_n \;|\; \D_i \in\dsep(H_g)_k \}
    \]
    We then let $\delta\dsep(H_g)_\bullet$ denote the diagonal simplicial set with $\delta\dsep(H_g)_n = \dsep_n(H_g)_n$.
\end{defn}

In both \cref{def: septop as simplicial set} and \cref{def: dseptop as simplicial set}, if the condition on $W_t$ or $\D_t$ is satisfied for one $t \in \Delta^k$ then it is satisfied for all.
Thus, the simplicial sets of separating systems are unions of path-components in $\mathrm{Sub}(U_g)_\bullet$ and $\mathrm{Sub}(H_g)_\bullet$, respectively, and from the discussion following \Cref{defn:lf-family} we see that the action of the respective homeomorphism groups on these simplicial sets induces Kan fibrations.

\begin{cor}\label{lem:homeo->sep is fibration}
    For every $\Sigma \in \septop_0(U_g)_0$ the map
    \[
        \SingHomeo(U_g) \longrightarrow \septop_0(U_g)_\bullet, \qquad
        \varphi \mapsto \varphi(\Sigma)
    \]
    is a Kan fibration.
    Similarly, the map $\SingHomeo(H_g) \to \dsep_0(H_g)$, defined by acting on a disc system, is a Kan fibration for any choice of disc system. 
\end{cor}

Every separating system $\Sigma$ (or family thereof) has a finite set of sub-separating systems $\Sigma' \subseteq \Sigma$.
Indeed, every sub-separating system is a union of components and thus determined by a choice of subset $\pi_0(\Sigma') \subset \pi_0(\Sigma)$.
Conversely, a union of components $\Sigma' \subseteq \Sigma$ is a separating system if and only if $U_g \ca \Sigma'$ has simply connected components.
The same description holds for $\Delta^k$-families of separating systems and thus we have a bijection between subsystems of $W \subset \Delta^k \times U_g$ and subsystems of the fiber $W_i \subset \{i\} \times U_g$ for any vertex $i \in \Delta^k$.
    Iterating this to describe $n$-chains of subsystems we see that every lifting problem 
    \[\begin{tikzcd}
        \Delta^0 \ar[d, "i"'] \ar[r] & \sep_n(U_g)_\bullet \ar[d, "d_0^n"] \\
        \Delta^k \ar[r] \ar[ur, dashed] & \sep_0(U_g)_\bullet
    \end{tikzcd}\]
    has a unique lift, where $d_0^n$ is the map that sends $(\Sigma_0 \subseteq \dots \subseteq \Sigma_n)$ to $\Sigma_n$.
In other words, $d_0^n\colon \sep_n(U_g) \to \sep_0(U_g)$ is a covering, \emph{i.e.}~it is $0$-coskeletal or, equivalently, it is a minimal Kan fibration with discrete fiber, see \cite[\S11]{May1992simplicial}.
Sub-disc systems in $H_g$ admit a similar characterisation, giving us the following lemma.

\begin{lem}\label{lem:coverings}
    In the diagram
    \[\begin{tikzcd}
        \dsep_n(H_g)_\bt \ar[r] \ar[d, "d_0^n"'] \ar[dr, phantom, very near start, "\lrcorner"] & 
        \septop_n(U_g)_\bt \ar[d, "d_0^n"] & 
        \Sing\sep_n(U_g) \ar[l] \ar[d, "d_0^n"] \ar[dl, phantom, very near start, "\llcorner"] \\
        \dsep_0(H_g)_\bt \ar[r] & 
        \septop_0(U_g)_\bt & 
        \Sing\sep_0(U_g) \ar[l] 
    \end{tikzcd}\]
    the vertical maps, defined by $(\Sigma_0 \subseteq \dots \subseteq \Sigma_n) \mapsto \Sigma_n$, are finite coverings, and both squares are pullback squares.
\end{lem}
\begin{proof}
    The left and middle vertical maps are coverings by the argument preceding the lemma. Moreover, they are finite coverings since every separating system has a finite set of sub-separating systems. 
    For the right vertical map, by \cite[Lemma 3.10]{BoydBregmanSteinebrunner-finiteness} $\sep_n(U_g) \to \sep_0(U_g)$ is a finite covering (of topological spaces).
    Applying $\Sing(-)$ yields a covering of simplicial sets in the above sense.
    The middle and right vertical maps have the same fibers so the right square is a pullback.

    For the left square, 
    note that for $\D \in \dsep(H_g)_0$ a union of components $\D' \subseteq \D$ is a sub separating system if and only if $\partial \D' \subset \partial \D$ is a subsystem.
    Therefore, in the left square the horizontal maps induce a bijection between the fibers of the vertical maps and thus this square is a pullback.
\end{proof}

To prove that $\delta(\septop(U_g))_\bullet$ is contractible we will compare it to the simplicial space of smooth separating systems from \cite{BoydBregmanSteinebrunner-finiteness}.
Recall that $\Diff(M, \Sigma) \subset \Diff(M)$ denotes the subgroup of diffeomorphisms that fix $\Sigma$ set-wise and similarly for $\Homeo(M, \Sigma) \subset \Homeo(M)$. 

\begin{lem}\label{lem:homeo=diff-of-pair}Let $M$ be a 3-manifold. 
    For every smooth separating system $\Sigma \subset M$ the map
    \[
       \Diff(M, \Sigma)  \longrightarrow  \Homeo(M, \Sigma)
    \]
    is a weak equivalence.
\end{lem}
\begin{proof}
  Consider the following map of fiber sequences of topological spaces:
    \[\begin{tikzcd}
	\Diff_\Sigma(M) & \Diff(M,\Sigma)  & \Diff(\Sigma) \\
	\Homeo_\Sigma(M) & \Homeo(M, \Sigma)  & \Homeo(\Sigma)
	\arrow[hook,from=1-1, to=2-1]
	\arrow[from=1-1, to=1-2]
        \arrow[from=1-2, to=1-3]
	\arrow[hook, from=1-2, to=2-2]
        \arrow[hook, from=1-3, to=2-3]
    \arrow[, from=2-1, to=2-2 ]
	\arrow[, from=2-2, to=2-3]
\end{tikzcd}\]
Here the top right map is a Serre fibration by \cite{Palais60, Cerf} and the bottom right map is a Serre fibration as a consequence of \cite[Theorem 4.14, p.129]{BurgheleaLashofRothenberg1975}.
(By \cite[Remark 16.5]{May1992simplicial} if $\Sing(f)$ is a Kan fibration, then $f$ is a Serre fibration.)
By the equivalence between homeomorphisms and diffeomorphisms for surfaces \cite{Rado,Epstein} the right map is an equivalence. As a consequence of work of Cerf \cite[\S3.2.1, Th\'eor\`eme 8]{Cerf} and Hatcher's resolution of the Smale conjecture \cite{Hatcher}, there is an equivalence of diffeomorphisms and homeomorphisms of 3-manifolds fixing a subsurface pointwise \cite{Cerf,Hatcher} so the left map is an equivalence.
It follows the central map is an equivalence, as required.
\end{proof}

\begin{prop}\label{prop:septop is contractible}
    The simplicial set $\delta (\septop (U_g))_\bt$ is weakly contractible.
\end{prop}
\begin{proof}
    We will show that for all $n$ the map
    \[  
        \Sing\sep_n(U_g) \longrightarrow \septop_n(U_g)_\bullet
    \]
    is a weak equivalence. 
    As recalled in Section 2 we then have
    \[
        |\delta(\septop(U_g))_\bullet| 
        \simeq \|[n] \mapsto |\septop_n(U_g)_\bullet|\|
        \simeq \|[n] \mapsto |\Sing \sep_n(U_g)| \|
        \simeq \|\sep_\bullet(U_g) \|
    \]
    which is contractible by \cite[Proposition 3.12]{BoydBregmanSteinebrunner-finiteness}.

    For simplicity, we first consider the case of $n=0$.
    At every $\Sigma \in \sep_0(U_g)_0$ consider the diagram
    \[\begin{tikzcd}
    	{\Sing\Diff(U_g,\Sigma)} & {\Sing\Diff(U_g)} & {\Sing\sep_0(U_g)} \\
    	{\SingHomeo(U_g,\Sigma)} & {\SingHomeo(U_g)} & {\septop_0(U_g)_\bullet}.
    	\arrow[from=1-1, to=1-2]
    	\arrow["\simeq", from=1-1, to=2-1]
    	\arrow[from=1-2, to=1-3]
    	\arrow["\simeq", from=1-2, to=2-2]
    	\arrow[from=1-3, to=2-3]
    	\arrow[from=2-1, to=2-2]
    	\arrow[from=2-2, to=2-3]
    \end{tikzcd}\]
    The top row is a fiber sequence as $\Diff(U_g) \to \sep_0(U_g)$ is a Serre fibration (see \cite[first line of the proof of Lemma 3.19]{BoydBregmanSteinebrunner-finiteness}) and thus becomes a Kan fibration after applying $\Sing$.
    It follows from \Cref{lem:homeo->sep is fibration} that the bottom sequence is also a fiber sequence.
    (Note, however, that the two Kan fibrations involved are usually not surjective.)
    In this diagram the middle map is an equivalence by \cite{Cerf,Hatcher} and the left map is an equivalence by \Cref{lem:homeo=diff-of-pair}.
    As we know this for all $\Sigma \in \sep_0(U_g)_0$ this shows that the right map is an equivalence onto the components it hits.

    We also need to argue that $\pi_0\sep_0(U_g) \to \pi_0(\septop_0(U_g)_\bullet)$ is surjective.
    If $\Sigma \in \septop_0(U_g)_0$ is a topological sphere system, then by \cite{Bing-Approximating-Surfaces} it is isotopic to a smooth sphere system.
    So far we have shown that $\Sing\sep_n(U_g) \to \septop_n(U_g)_\bullet$ is a weak equivalence when $n=0$.
    It follows from the right pullback square in \Cref{lem:coverings} that it is a weak equivalence for all $n$, completing the proof.
\end{proof}

\section{Proof of the main theorem}\label{sec:proof}

We need the following adaptation of the Alexander trick.
\begin{lem}\label{lem: conical Alexander trick}
Let $V\subset S^3=\partial D^4$ be a submanifold.
The map
\[\Homeo_{V}(D^4)\overset{\partial}{\to} \Homeo_{V}(S^3)\] given by restriction to the boundary is a homotopy equivalence.
\end{lem}
\begin{proof}
    We check that the proof of the Alexander trick, which shows $\Homeo(D^4)\simeq \Homeo(S^3)$ restricts to these subgroups. 
    We first define a map 
    \begin{align*}
        s\colon \Homeo_{V}(S^3) &\overset{\partial}{\to} \Homeo_{V}(D^4)\\
        \varphi &\mapsto \left(x \mapsto |x|\cdot \varphi\big(\tfrac{x}{|x|}\big)\right).
    \end{align*}
    Then $\partial \circ s = \operatorname{id}_{\Homeo_{V}(S^3)}$, and $s\circ \partial \simeq \operatorname{id}_{\Homeo_{V}(D^4)}$ via the following homotopy.
    \begin{align*}
        H\colon \Homeo_{V}(D^4) \times I &\to \Homeo_{V}(D^4)\\
        (\psi, t) &\mapsto \begin{cases}
            |x|\cdot \psi \big(\frac{x}{|x|}\big) & |x|\geq t \\
            t\cdot \psi\big(\frac{x}{t}\big) & |x|\leq t.
        \end{cases}
    \end{align*}
    Note that throughout this homotopy the homeomorphism on the boundary remains unchanged, and so $V\subset \partial M$ remains pointwise fixed, \emph{i.e.}~we stay in the required subgroup. 
\end{proof}

\begin{lem}\label{lem:stabiliser-equivalence}
Let $\D\in \dsep(H_g)$ be a disc system. Then
    \[\Homeo(H_g,\D) \overset{\partial}{\to} \Homeo(U_g, \partial \D)\] is a weak equivalence.
\end{lem}
\begin{proof}
    Consider the following map of fiber sequences,
    given by restriction to the boundary.
    (The right horizontal maps are Serre fibrations by the argument in \Cref{lem:homeo=diff-of-pair}.)
    \[\begin{tikzcd}
	\Homeo_\D(H_g) & \Homeo(H_g,\D)  & \Homeo(\D) \\
	\Homeo_{\partial\D}(U_g) & \Homeo(U_g,\partial\D)  & \Homeo(\partial\D) 
	\arrow["\partial",from=1-1, to=2-1]
	\arrow[from=1-1, to=1-2]
    \arrow[from=1-2, to=1-3]
	\arrow["\partial", from=1-2, to=2-2]\arrow["\partial", from=1-3, to=2-3]
	\arrow[, from=2-1, to=2-2]
	\arrow[, from=2-2, to=2-3]
\end{tikzcd}\]
The right hand vertical map is an equivalence by the Alexander trick in dimension 3 ($\Homeo(D^3)\to \Homeo(S^2)$ is an equivalence). The domain and codomain of the left hand vertical map can be rewritten as
\[
\Homeo_\D(H_g)\cong \prod \Homeo_{\sqcup D^3}(D^4) 
\qquad\text{and}\qquad
\Homeo_{\partial\D}(U_g)\cong \prod \Homeo_{\sqcup D^3}(S^3).
\]
where the number of factors in the product are the same and correspond to components of $H_g\ca \D$ and $U_g\ca \partial \D$, respectively. Since $\sqcup D^3\subset \partial D^4$, we apply \Cref{lem: conical Alexander trick} to each component. Therefore the left hand vertical map is also an equivalence, hence the central vertical map is an equivalence, as required.
\end{proof}

\begin{lem}\label{lem:disk system for each dual graph}
    Let $\Sigma \in \septop(U_g)_0$ be a separating system in $U_g$.
    Then there is a disc system $\D \in \dsep(H_g)_0$ and $\phi \in \Homeo(U_g)$ with $\phi(\partial \D) = \Sigma$.
\end{lem}
\begin{proof}
     Consider the dual graph $\Gamma_\Sigma$ of $\Sigma \subset U_g$, whose vertices correspond to the components of $U_g \ca \Sigma$, and whose edges correspond to the spheres in $\Sigma$. We first build a manifold $N$ that is homeomorphic to $H_g$ and contains a disc system $\D$ such that $\Gamma_{\partial \D}\cong \Gamma_\Sigma$, as follows. To each $v\in V(\Gamma_\Sigma)$ assign a disc $D^4_v$, to obtain the manifold $\coprod_{v\in V(\Gamma_\Sigma)} D^4_v$. 
     Now for each edge $e=\{v,w\}$ in $E(\Gamma_\Sigma)$ choose a standard 3-disc in the boundaries of the 4-discs $D^4_v$ and $D^4_w$, respectively, and glue on a handle $h_e=D_e^3\times I$, in such a way that the attaching discs are pairwise disjoint. The output of this construction is a topological manifold $N$ with boundary that is homeomorphic to $H_g$. Let $\psi\colon N\cong H_g$ be a choice of such a homeomorphism. The required disc system $\D\subset H_g$ is given by $\sqcup_{e\in E(\Gamma_\Sigma)}\psi(D_e^3\times \{\frac{1}{2}\})$.

    It remains to find $\phi \in \Homeo(U_g)$ with $\phi(\partial \D) = \Sigma$. By construction, we have an identification of dual graphs $f\colon\Gamma_\Sigma\cong \Gamma_{\partial \D}$ giving a bijection between components $\pi_0(\Sigma)$ and $\pi_0(\partial \D)$ (edges), and between components $\pi_0(U_g \ca \Sigma)$ and $\pi_0(U_g \ca \partial \D)$ (vertices).
    We build $\phi \in \Homeo(U_g)$ via the following two steps.
    \begin{enumerate}
        \item Each component $K \in \pi_0(U_g \ca \Sigma)$ is homeomorphic to the component $f(K)=K' \in \pi_0(U_g \ca \partial \D)$, and to $S^3\setminus \sqcup_m \interior{D}^3$, where $m$ is the valence of the corresponding vertex in the dual graph. Furthermore, each boundary component of $K$ (resp.~$K'$) is identified with a sphere in $\Sigma$ (resp.~$\partial D$). Using this identification, pick a homeomorphism $\phi_K\colon K \to K'$ such that $\phi_K(S)=f(S)\in \pi_0(\partial \D)$ for all $S\in \Sigma$ (this is always possible since $\Homeo(S^3)$ acts transitively on embedded discs).
        \item We now make the homeomorphisms $\{\phi_K\}_{K\in \pi_0(U_g \,\ca\, \Sigma)}$ compatible so that they can be glued along $\Sigma$ in the domain and $\partial \D$ in the codomain to yield the desired $\phi \in \Homeo(U_g)$. Given a sphere $S\subset \Sigma$ there are components $K_1,K_2$ of $U_g\ca \Sigma$ and  inclusions $\rho_i\colon S\rightarrow S_i\subset \partial K_i$ for $i=1,2$. From Step $(1)$, we have homeomorphisms $\phi_{K_i}\colon K_i\rightarrow K_i'$ such that $\phi_{K_i}(S_i)$  corresponds to the sphere $f(S)\in \partial \D$ under the inclusions $\rho_i'\colon f(S)\rightarrow K_i'$. Since $\Homeo^+(S^2)$ is path-connected we can isotope $\phi_{K_1}$ in a neighbourhood of $S_1$ so that $(\rho_1')^{-1}\circ\phi_{K_1}\circ\rho_1=(\rho_2')^{-1}\circ\phi_{K_2}\circ \rho_2.$ After doing this for each $S\in \Sigma$, the $\phi_K$ agree along each component of $\Sigma$ hence we obtain our desired $\phi\in \Homeo(U_g)$.\qedhere
    \end{enumerate}
\end{proof}

\begin{thm}\label{thm:section}
    The map
    $\partial\colon \BHomeo(H_g) \to \BHomeo(U_g)$
    has a section.
\end{thm}
\begin{proof}
    As recalled in Section 2, we can model classifying spaces as the realisation of the simplicial homotopy orbit construction and so it suffices to construct a section (up to homotopy) of the map
    \[
        |*\hq\SingHomeo(H_g)| \xrightarrow[\quad]{\partial} |*\hq\SingHomeo(U_g)|.
    \]

    We will apply \Cref{lem:orbit-stabiliser} to the map $\dsep_0(H_g)_\bt \to \septop_0(U_g)_\bt$ to prove that the map
    \[
        \dsep_0(H_g)_\bt \hq \SingHomeo(H_g) \longrightarrow
        \septop_0(U_g)_\bt \hq \SingHomeo(U_g)
    \]
    is a weak equivalence.
    By \Cref{lem:homeo->sep is fibration} both of the action maps $\SingHomeo(H_g) \to \dsep_0(H_g)_\bt$ and $\SingHomeo(U_g) \to \septop_0(U_g)_\bt$ are fibrations.
    The map of stabiliser groups is 
    \[
        \SingHomeo(H_g, \D) \longrightarrow \SingHomeo(U_g, \partial \D)
    \]
    which is indeed a weak equivalence by \Cref{lem:stabiliser-equivalence}.
    Moreover, by \Cref{lem:disk system for each dual graph}, every $\Sigma \in \septop_0(U_g)_0$ can be written as $\phi(\partial \D)$ for some $\D \in \dsep_0(H_g)_0$ and $\phi \in \Homeo(U_g)$.
    This verifies the surjectivity condition of \Cref{lem:orbit-stabiliser}, so we have
    \[
        \dsep_n(H_g)_\bt \hq \SingHomeo(H_g) \longrightarrow
        \septop_n(U_g)_\bt \hq \SingHomeo(U_g)
    \]
    is a weak equivalence for $n=0$.
    Because the left square in \Cref{lem:coverings} remains a pullback after taking homotopy orbits, the case of general $n$ follows.
    
    We have thus shown that the left hand map in the diagram
    \[\begin{tikzcd}
    	\delta(\dsep(H_g))_\bt \hq \SingHomeo(H_g) & *\hq\SingHomeo(H_g) \\
    	\delta(\septop(U_g))_\bt \hq \SingHomeo(U_g)& *\hq \SingHomeo(U_g).
    	\arrow["\partial", "\simeq"', from=1-1, to=2-1]
    	\arrow[from=1-1, to=1-2]
    	\arrow["\partial", from=1-2, to=2-2]
    	\arrow["\simeq", from=2-1, to=2-2]
    \end{tikzcd}\]
    is a weak equivalence, as depicted.
    The bottom map is a weak equivalence as $\delta(\septop(U_g))_\bt \simeq \delta (\Sing\sep(U_g))_\bt \simeq *$ by \Cref{prop:septop is contractible}.
    After passing to geometric realisations we can find homotopy inverses to these equivalences and the desired lift is then given by starting at the bottom right and going left, up, and right.
\end{proof}

We finish with the proof of \Cref{cor:section}. In particular, applying this inductively gives the analogous result with multiple $U_g$ boundary components.

\begin{cor}\label{cor:section}
    Let $M$ be a compact $4$-manifold with $U_g \subseteq \partial M$.
    Then the restriction map
    \[
        \BHomeo(M \cup_{U_g} H_g, H_g) \longrightarrow \BHomeo(M, U_g)
    \]
    admits a section.
\end{cor}
\begin{proof}
    The restriction maps between homeomorphism groups induce a square of classifying spaces
    \[\begin{tikzcd}
        {\BHomeo(M \cup_{U_g} H_g, H_g)} \ar[r] \ar[d] \ar[dr, very near start, phantom, "\lrcorner"] &
        {\BHomeo(H_g)} \ar[d] \\
        {\BHomeo(M, U_g)} \ar[r] \ar[u, bend left, dashed] & 
        {\BHomeo(U_g)}. \ar[u, bend left, "s"]
    \end{tikzcd}\]
    The vertical maps are surjective on $\pi_1$ by \Cref{thm:section} (or \cite{LaudenbachPoenaru1972}). 
    This is a homotopy pullback square as the induced map on vertical fibers is
    \[
        \BHomeo_{M}(M \cup_{U_g} H_g) \longrightarrow 
        \BHomeo_{U_g}(H_g),
    \]
    which is an equivalence.
    The right vertical map in the square admits a section by \Cref{thm:section} and thus we can pull it back to obtain a section of the left vertical map.
\end{proof}

\bibliography{mybib}{}
\bibliographystyle{alpha}
\end{document}